\RequirePackage{fix-cm}
\documentclass[12pt]{article}
\usepackage{amsmath, amssymb, amsfonts, amsxtra, latexsym, amscd, eucal, graphicx,stix, amsthm}
\usepackage{xcolor}      % use if color is used in text
\usepackage{hyperref}   % use for hypertext links, including those to external documents and URLs

\theoremstyle{definition}
\newtheorem{definition}{Definition}[section]

\newtheorem{remark}{Remark}[section]

\newtheorem{lemma}[definition]{Lemma}
\newtheorem{theorem}[definition]{Theorem}

\newcommand{\mf}{\mathcal{F}}

\newcommand{\pr}{\mathbb{P}}

\newcommand{\bE}{\mathbb{E}}

\newcommand{\Om}{\Omega}

\newcommand{\rd}{d}

\usepackage[top=2.5cm, bottom=2.5cm, left=2.5cm, right=2cm]{geometry}

\allowdisplaybreaks

\begin{document}

\title{Semi-implicit Milstein approximation scheme for non-colliding particle systems\footnote{This research is funded by Vietnam National Foundation for Science and Technology Development (NAFOSTED) under grant number 101.03-2017.316. The paper was completed during a scientific stay
		of the second author at the Vietnam Institute for Advanced Study in Mathematics (VIASM), whose
		hospitality is gratefully appreciated.
}}
\author{   Duc-Trong Luong\footnote{Email: trongld@hnue.edu.vn} \quad    Hoang-Long Ngo\footnote{Eamil: ngolong@hnue.edu.vn}}
\date{Hanoi National University of Education} 
%	Duc-Trong Luong  \at
%	Hanoi National University of Education, 136 Xuan Thuy, Cau Giay, Hanoi, Vietnam \\
%	\email{trongld@hnue.edu.vn}           %  \\
%	%             \emph{Present address:} of F. Author  %  if needed
%	\and
%	Hoang-Long Ngo  \at
%	Hanoi National University of Education, 136 Xuan Thuy, Cau Giay, Hanoi, Vietnam \\
%	\email{ngolong@hnue.edu.vn}           %  \\
%	%             \emph{Present address:} of F. Author  %  if needed
%	        %  \\
%	%             \emph{Present address:} of F. Author  %  if needed
%}


\maketitle

\textbf{Abstract:} We introduce a semi-implicit Milstein approximation scheme for some classes of non-colliding particle systems modeled by systems of stochastic differential equations with non-constant diffusion coefficients. We show that the scheme converges at the rate of order 1 in the mean-square sense.


%\subclass{65C30 \and  60H35 \and  60K35}
%%%%%%%%%%%%%%%%%%%%%%%%%%%%%%%%%%%%%%%%%%%%%%%%%%%%%%%%%%%%%%%%%%%%%%%%%%%%%%
\section{Introduction}
 We consider a process $X= (X_1(t),X_2(t),\ldots,X_d(t))_{t\geq 0}$ given by the following stochastic differential equation (SDE)
 \begin{equation} \label{Xt}
X_i(t)=X_i(0) + \int_0^t \left(\sum\limits_{j\neq i} \dfrac{\gamma_{ij}}{X_i(s)-X_j(s)} +b_i(X_i(s))\right)ds
+ \int_0^t \sigma_i(X_i(s))dW_i(s), \quad {\color{black} 1\leq i\leq d,}
\end{equation}
where {\color{black}$X(0)$ is a deterministic constant  and} belongs to $\Delta_d = \{\mathbf{x}=(x_1,x_2,\ldots,x_d) \in \mathbb{R}^d:x_1<x_2<\ldots<x_d\},$  $\gamma_{ij} =\gamma_{ji}\geq 0$ and $({\color{black}W(t)}=(W_1(t),W_2(t)),\ldots,W_d(t))_{t\geq 0}$ is a $d$-dimensional Brownian motion defined on a filtered probability space  $(\Om, \mf, (\mf_t)_{t\geq 0},\pr)$.

In mathematical physics, the process {\color{black}$X$} is used to model systems of $d$ non-colliding particles evolving on the real line, such as Dyson Brownian motion or particles with electrostatic repulsion. The SDE  \eqref{Xt} was first studied by {\color{black} Dyson (1962), where it} is used to represent the eigenvalues of a $d\times d$-dimensional symmetric Gaussian random matrix. The theory was later developed by Bru (1989) and Bru (1991), where it was showed that the eigenvalues of a Wishart process also satisfy a system of the form \eqref{Xt}. There have been many works on the existence and uniqueness of the solution to equation \eqref{Xt}, e.g., C\'epa and L\'epingle  (1997), Graczyk and Ma\l ecki (2014), L\'epingle (2010), Rogers and Shi (1993), Nakanuma and Taguchi (2018). Many applications and interesting {\color{black} features} of {\color{black}$X$} were presented in Katori and Tanemura (2004), Rost and Vares (1985), and  Ramanan and Shkolnikov (2018).  

The main aim of this paper is to introduce a high order numerical approximation scheme for equation \eqref{Xt} such that the approximate solution always stays in $\Delta_d$. Since the multidimensional SDEs whose solution stays in a domain appear in many applications such as biology, finance, and physics (see Kloeden and Platen 1995), their numerical approximation has been studied extensively.  Gy\"ongy (1998) introduced a polygonal Euler approximation for SDEs on domains of
$\mathbb{R}^d$ and showed that it converges almost surely if the drift {\color{black} coefficient} satisfies a monotonicity condition and the diffusion coefficient is Lipschitz continuous. For SDEs with locally Lipschitz continuous coefficients,  Jentzen et al. (2009) introduced a projection Euler method and showed that it converges at the rate of order $1$ in the pathwise sense. The main difficulty in constructing a numerical approximation for equation \eqref{Xt} comes from the fact that its drift coefficient is non-locally Lipschitz continuous and even blows up at the boundary of $\Delta_d$. 
The first numerical simulation for $X_t$ is presented in Li and Menon (2013) where the authors introduced a tamed Euler-Maruyama approximation scheme. However, this tamed scheme  does not preserve the non-colliding property of the original system.   Ngo and Taguchi (2017) introduced a semi-implicit Euler-Maruyama approximation scheme for the SDE \eqref{Xt} and studied its convergence in $L^p$-norm. A key feature of their new scheme is that the approximate solution always stays inside the domain $\Delta_d$ as the true solution does.  They showed that if the {\color{black}coefficients} $b= (b_i)_{1 \leq i \leq d }$ and $\sigma=(\sigma_i)_{1 \leq i \leq d}$ are Lipschitz continuous then the Euler-Maruyama approximation scheme converges at the rate of order $1/2$. Moreover, if $\sigma$ is a constant and $b$ is differentiable up to order $2$, then the Euler-Maruyama approximation scheme converges at the rate of order $1$. 

In this paper, we introduce a semi-implicit Milstein approximation scheme for the SDE \eqref{Xt}. We show that the approximate solution always stays inside the domain $\Delta_d$ and it converges at the rate of order $1$ in the mean-square sense when $b$ and $\sigma$ are bounded and differentiable continuous up to order $2$. Since when $\sigma$ is constant, our semi-implicit Milstein scheme coincides with the semi-implicit Euler-Maruyama scheme in Ngo and Taguchi (2017), our result can be considered as a generalization of the one in Ngo and Taguchi (2017) for SDEs with non-constant diffusion coefficients. To the best of our knowledge, this is the first approximation scheme of strong order 1 for multidimensional SDEs defined in a domain.  

The rest of the paper is organized as follows. In Section \ref{sec2}, we introduce the semi-implicit Milstein approximation scheme and state our main result in Theorem \ref{Thm1}. The proof is given in Section \ref{sec3}. A numerical simulation is presented in Section \ref{sec4} 

 \section{Semi-implicit Milstein approximation scheme} \label{sec2}
 The semi-implicit Milstein approximation scheme is defined as follows. For each integer $n\geq 1$ and $T>0$, we set  $t^{(n)}_k = \frac{kT}{n}$, and 
  $X^{(n)}(0):=X(0)$, and for each $k=0,\ldots, n-1$ and ${\color{black} t \in \left[t^{(n)}_{k},t^{(n)}_{k+1}\right]}$, $X^{(n)}(t)= (X^{n}_i(t))_{1\leq i \leq d}$ is the unique solution in  $\Delta_d$ of the following equations
 \begin{align}\label{Xnt}
 X_i^{(n)}(t)&= X_i^{(n)}(t^{(n)}_{k})+ \left[\sum\limits_{j\neq i}\dfrac{\gamma_{ij}}{X_i^{(n)}(t)-X_j^{(n)}(t)}+b_i\left(X_i^{(n)}(t^{(n)}_{k})\right)\right](t - t^{(n)}_k)\notag\\
 &+ \sigma_{i}\left(X_i^{(n)}(t^{(n)}_{k})\right)\left[W_i(t)-W_i(t^{(n)}_{k})\right]\notag\\
 &+\dfrac{1}{2}{\color{black}\sigma_{i}\left(X_i^{(n)}(t^{(n)}_{k})\right)\sigma'_{i}\left(X_i^{(n)}(t^{(n)}_{k})\right)}\left[\left(W_i(t)-W_i(t^{(n)}_{k})\right)^2-(t - t^{(n)}_k)\right], \quad i=1,\ldots, d. 
 \end{align}
The existence and uniqueness of solution to equation \eqref{Xnt} {\color{black} follows} from Proposition 2.2 in Ngo and Taguchi (2017)  under an assumption that $\gamma_{i,i+1}>0$ for all $i=1,\ldots, d-1$. 

\indent We {\color{black} set}
 $X_{ij}(t)=X_i(t)-X_j(t), X^{(n)}_{ij}(t)=X^{(n)}_i(t)-X^{(n)}_j(t)$ and  $e_i(t)=X_i(t)-X_i^{(n)}(t)$. For $x\in \mathbb{R}^d$, we denote by $\|x\|$ the Euclidian norm of $x$.

{\color{black} Throughout this paper, we use $C>0$ to denote a generic constant, which is independent of $n$, but may depend on $b,\sigma,\gamma_{ij}$ and $x_0$. The value of $C$ may vary frome place to place. When $C$ depends on some addtional parameter, say $p$, we denote it by $C(p)$.} 
%\begin{Ass}\label{gt}

\noindent 
\textbf{Assumption $H_{\hat{p}}$:}
	The equation \eqref{Xt} has a unique strong solution in $\Delta_d$, and there exist some {\color{black}positive constants  $\hat{p}$ and  $C$ such that 
	$$\sup\limits_{t\in[0,T]}\mathbb{E}\left[\|X(t)\|^{\hat{p}}\right]+\max\limits_{1\leq i\leq d-1}\sup\limits_{t\in[0,T]}\mathbb{E}\left[|X_{i,i+1}(t)|^{-\hat{p}}\right]<C,$$
	and 
	$$\mathbb{E}\left[\|X(t)-X(s)\|^{\hat{p}}\right]\leq C|t-s|^{\hat{p}/2},\mbox{ for all } 0\leq s<t\leq T.$$}
%\end{Ass}

\begin{remark}
	It was shown in Ngo and Taguchi (2017) that  Assumption $H_{\hat{p}}$ is satisfied for some classes of particle systems of the form \eqref{Xt}, such as the interacting Brownian particles and the Brownian particles with nearest neighbor replusion. 
\end{remark} 
	
%\section{Kết quả chính}
{\color{black}We denote by $\mathcal{T}_n$ the set of all stopping times $\tau$ taking value in the set $\{t^{(n)}_k, 0\leq k \leq n\}$, and $C^2_b(\mathbb{R})$} the set of all functions $f: \mathbb{R} \to \mathbb{R}$ such that $f, f'$ and $f''$ are bounded. 
\begin{theorem}\label{Thm1}
	Suppose that $b,\sigma\in C^2_b(\mathbb{R})$.
	{\begin{enumerate}
			\item[(i)] If Assumption $H_{\hat{p}}$ holds for some  $\hat{p}\geq 6$, then {\color{black}
			%there exists a positive constant $C_1$ 	which does not depend on $n$ such that 
			\begin{align} \label{kq1}
			\sup_{\tau \in \mathcal{T}_n} \mathbb{E}\left[ \left\|X(\tau)-{\color{black}X^{(n)}}(\tau)\right\|^2 \right] \leq \dfrac{C}{n^2}.
			\end{align}
			Moreover, for any $p \in (0,2)$, it holds that
			%there exists a positive constant $C_1(p)$ which does not depend on $n$ such that 
			\begin{align} \label{kq2} \mathbb{E}\left[\sup\limits_{0\leq k\leq n}\|X(t^{(n)}_k)-X^{(n)}(t^{(n)}_k)\|^p\right]\leq \dfrac{C(p)}{n^p}.
			\end{align}}
			\item[(ii)] If Assumption $H_{\hat{p}}$ holds for some  ${\hat{p}\geq 18}$, then for any $p \in (0,2)$, 
			%there exists a positive constant $C_2(p)$ which does not depend on $n$ such that 
			\begin{equation} \label{kq3}
			\mathbb{E}\left[\sup_{t\in[0,T]}\left\|X(t)-X^{(n)}(t)\right\|^p\right]\leq C(p)\dfrac{(\log n)^{3p/2}}{n^p}.
			\end{equation}
	\end{enumerate}}
\end{theorem}

\section{Proof} \label{sec3} 

\subsection{Representation of estimate error} 
	For each $i = 1, \ldots, d$, we {\color{black} denote by}  
	$e_i(t)=X_i(t)-X_i^{(n)}(t),$ and  $\|e(t)\|^2=\sum\limits_{i=1}^{d}e_i^2(t).$
	For each {\color{black}$t\in\left[t_k^{(n)},t_{k+1}^{(n)}\right]$ and $i=1, \ldots, d$}, it follows from \eqref{Xt} and \eqref{Xnt} that   	
	\begin{align} \label{etetn}
	e_i(t) &=e_i(t^{(n)}_{k})+V_i(t)+\sum\limits_{j\neq i}\left(\dfrac{\gamma_{ij}}{X_{ij}(t)}-\dfrac{\gamma_{ij}}{X^{(n)}_{ij}(t)}\right)(t-t^{(n)}_k), 
	\end{align}
	where $V_i(t)=S_{1i}(t)+S_{2i}(t)+S_{3i}(t)+S_{4i}(t)+S_{5i}(t){\color{black}+S_{6i}(t)}$, and 
\begin{align}
	S_{1i}(t) &=  \int_{t^{(n)}_{k}}^{t}\sum\limits_{j\neq i}\left(\dfrac{\gamma_{ij}}{X_{ij}(s)}-\dfrac{\gamma_{ij}}{X_{ij}(t)}\right)ds, \label{dnS1i} \\
	S_{2i}(t) &= \int_{t^{(n)}_{k}}^{t}\left[b_i(X_i(s))-b_i(X_i(t^{(n)}_k))\right]ds, \label{dnS2i}\\
	S_{3i}(t) &= \int_{t^{(n)}_{k}}^{t} \left[\sigma_i(X_i(s))-\sigma_i(X_i(t^{(n)}_k))-\int_{t^{(n)}_{k}}^{s}{\color{black}\sigma_i \left(X_i(t^{(n)}_{k})\right)\sigma_i' \left(X_i(t^{(n)}_{k})\right)} dW_i(u)\right] dW_i(s), \label{dnS3i}\\
	S_{4i}(t) &= \left[b_i(X_i(t^{(n)}_k))-b_i(X_i^{(n)}(t^{(n)}_k))\right](t-t_k^{(n)}), \label{dnS4i}\\
	S_{5i}(t)  &= \left[\sigma_i({\color{black}X_i(t^{(n)}_k)})-\sigma_i(X_i^{(n)}(t^{(n)}_k))\right]\left(W_i(t)-W_i(t^{(n)}_k)\right), \label{dnS5i}\\
	{\color{black}S_{6i}(t)}  &= \dfrac{1}{2}\left[\sigma_i(X_i(t^{(n)}_k))\sigma'_i(X_i(t^{(n)}_k))-\sigma_i(X_i^{(n)}(t^{(n)}_k))\sigma'_i(X_i^{(n)}(t^{(n)}_k))\right] \notag \\
		&\qquad \qquad  \times \left[\left(W_i(t)-W_i(t_k^{(n)}\right)^2-(t-t_k^{(n)}\right]. \label{dnS6i}
\end{align}
It follows from \eqref{etetn} that   
	\begin{align*}
	{\color{black}e_i^2(t_k^{(n)})}&+2V_i(t)e_i(t^{(n)}_k)+V_i^2(t)=\left\{e_i(t)-\sum\limits_{\color{black}j\neq i}\left[\dfrac{\gamma_{ij}}{X_{ij}(t)}-\dfrac{\gamma_{ij}}{X^{(n)}_{ij}(t)}\right](t-t^{(n)}_k)\right\}^2\\
	&\geq {\color{black}e_i^2(t)}-2e_i(t) \sum\limits_{j\neq i} \left[\dfrac{\gamma_{ij}}{X_{ij}(t)}-\dfrac{\gamma_{ij}}{X^{(n)}_{ij}(t)}\right](t-t^{(n)}_k).
	\end{align*}
	This implies that 
	\begin{align*}
	\|e(t_k^{(n)})\|^2&+2\sum\limits_{i=1}^d V_i(t)e_i(t^{(n)}_k)+\sum\limits_{i=1}^dV_i^2(t)\\
	&\geq {\color{black}\|e(t)\|^2}-2(t-t^{(n)}_k)\sum\limits_{i=1}^d e_i(t)\sum\limits_{j\neq i}\left[\dfrac{\gamma_{ij}}{X_{ij}(t)}-\dfrac{\gamma_{ij}}{X^{(n)}_{ij}(t)}\right] \notag \\
	&= \|e(t)\|^2-(t-t^{(n)}_k)\sum\limits_{i=1}^d\sum\limits_{j\neq i}\left(e_i(t)-e_j(t)\right)\left[\dfrac{\gamma_{ij}}{X_{ij}(t)}-\dfrac{\gamma_{ij}}{X^{(n)}_{ij}(t)}\right] \notag \\
	&= \|e(t)\|^2-(t-t^{(n)}_k)\sum\limits_{i=1}^d\sum\limits_{j\neq i}\left(X_{ij}(t)-X_{ij}^{(n)}(t)\right)\left[\dfrac{\gamma_{ij}}{X_{ij}(t)}-\dfrac{\gamma_{ij}}{X^{(n)}_{ij}(t)}\right] \notag \\
	&\geq \|e(t)\|^2, \label{et}
	\end{align*}
	{\color{black} where the last estimate follows from the fact that 
	$(X_{ij}(t)-X_{ij}^{(n)}(t))(\dfrac{\gamma_{ij}}{X_{ij}(t)}-\dfrac{\gamma_{ij}}{X^{(n)}_{ij}(t)})\leq 0$ for any $ i\neq j.$}
Therefore,  
{\color{black}\begin{align} 
\|e(t)\|^2&\leq \|e(t_k^{(n)})\|^2 +6 {\sum_{m=1}^6 \|S_{m}(t)\|^2} +2 \sum_{m=1}^{6}R_m(t),
\end{align}}
where 
\begin{equation} \label{dnR}
\|S_m(t)\|^2=\sum_{i=1}^dS_{mi}^2(t),\mbox{ and } R_m(t) = \sum_{i=1}^d e_i(t^{(n)}_k)S_{mi}(t).
\end{equation} 
In the following we will estimate the expectations of $S^2_{mi}$ and $R_m$ for $m = 1,2,\ldots, 6$, and $i=1,2, \ldots, d.$

\subsection{Some auxiliary estimates}
We need the following simple estimate. 
\begin{lemma}\label{Lem1}
	Let {\color{black}$(a_k)_{0 \leq k\leq n}$,$(\zeta_k)_{0 \leq k \leq n}$ and $(\xi_k)_{0 \leq k \leq n}$} are adapted processes defined on a filtered probability space $(\Omega,\mathcal{G}, (\mathcal{G}_k)_{0\leq k \leq n},\mathbb{P})$ such that
	\begin{enumerate}
		\item[(i)]  $a_0=0$ and $a_k\geq 0$ for any $1 \leq k \leq n$, 
		\item[(ii)] $\mathbb{E}(\xi_{k+1}|\mathcal{G}_k)=0$, for any $0 \leq k \leq n-1$, 
		\item[(iii)] $a_{k+1}\leq qa_k+\zeta_k+\xi_{k+1}$ for any $0 \leq  k \leq n-1$, for some $q>1$,
		\item[(iv)]  $\sup_{0\leq k\leq n}\mathbb{E}[|\zeta_k|]\leq \varepsilon$ for some $\varepsilon>0$. 
	\end{enumerate}
	Then for any stopping time $\tau \leq n$,
	\begin{equation*}
	\mathbb{E}[a_{\tau}]\leq \dfrac{\varepsilon q^n}{q-1}.
	\end{equation*}
\end{lemma}
\begin{proof}
	It follows from condition (iii) that 
$\sum_{i=0}^{k} q^{k-i} a_{i+1} \leq \sum_{i=0}^{k} \Big( q^{k-i+1}a_i + q^{k-i}\zeta_i + {\color{black}q^{k-i}}\xi_{i+1} \Big).$
	This fact together with condition (i) implies 
$a_{k+1}\leq \sum\limits_{i=0}^{k} q^{k-i} \zeta_{i}+\sum\limits_{i=1}^{k+1} q^{k+1-i} \xi_i. $\\
	It leads to $q^{-k}a_k \leq \sum\limits_{i=0}^{n-1} q^{-i-1} |\zeta_{i}|+\sum\limits_{i=1}^{k} q^{-i} \xi_i.$
	Let  $M_k=\sum\limits_{i=1}^{k} q^{-i} \xi_i$.  For all stoping time $\tau\leq n$, 
	$q^{-n} {a_\tau}\leq {q^{-\tau}} {a_\tau}\leq \sum\limits_{i=0}^{n-1} q^{-i-1} |\zeta_{i}|+M_\tau.$
	Thanks to condition (ii), $(M_k, \mathcal{G}_k)_{1\leq k\leq n}$ is a martingale. Using condition (iv) and Doob's optional sampling theorem, we get   
$	\mathbb{E}[a_\tau]\leq \sum\limits_{i=0}^{n-1}q^{n-1-i}\varepsilon\leq \dfrac{\varepsilon q^n}{q-1},$
	which implies the desired result.
\end{proof}

We also need the following moment estimates for $X$ and its modulus of continuity. 
\begin{lemma}\label{bd2}
%	Suppose that $b,\sigma\in C^2_b(\mathbb{R})$.
	\begin{enumerate}
			\item[(i)] Let Assumption $H_{\hat{p}}$ hold for some  $\hat{p}\geq 2$, then 
			%{\color{black}there exist positive constants $\bar{C}_1,\bar{C}_2$, 	which do not depend on $n$ such that }
				\begin{align}   
			&\mathbb{E}\left[\sup\limits_{t \in [0,T]}\|X(t)\|^{\hat{p}}\right] < C(\hat{p}), \label{bd2kq1} \\
			&\mathbb{E}\left[\sup_{s\leq t, t'\leq s'}\|X(t)-X(t')\|^{\hat{p}}\right] \leq C(\hat{p}) \left(|s'-s|\ln\dfrac{2T}{|s'-s|}\right)^{\hat{p}/2},\mbox{ for all } 0\leq s<s'\leq T. \label{bd2kq3}
			\end{align} 
			\item[(ii)] Let Assumption $H_{\hat{p}}$ hold for some  $\hat{p}\geq 3${\color{black} , then} 
			\begin{equation}
			\mathbb{E}\left[\max\limits_{i=1,\ldots,d-1}\sup\limits_{t \in [0,T]} |X_{i,i+1}(t)|^{-\hat{p}/3}\right]<C(\hat{p}),\label{bd2kq2}
			\end{equation} 
	\end{enumerate}
\end{lemma}
\begin{proof}
	(i) \  Let Assumption $H_{\hat{p}}$ hold for $\hat{p} \geq 2$. Since $b_i$ is bounded, 
	\begin{align*}
	|X_i(t)|^{\hat{p}}
	&\leq C(\hat{p}) + C(\hat{p})\int_0^T\sum\limits_{j\neq i} \left|\dfrac{\gamma_{ij}}{X_{ij}(s)}\right|^{\hat{p}}ds +C(\hat{p})\left|\int_0^t \sigma_{i}(X_i(s))dW_i(s)\right|^{\hat{p}}.
	\end{align*}
	Thanks to   Burkholder-Davis-Gundy's inequality, we get 
	\begin{align*}
	{\color{black}\mathbb{E}\left[\sup\limits_{t \in [0,T]} |X_i(t)|^{\hat{p}}\right]}&\leq C(\hat{p})+C(\hat{p})\mathbb{E}\left[\sup\limits_{t \in [0,T]} \left|\int_0^t \sigma_{i}(X_i(s))dW_i(s)\right|^{\hat{p}}\right]\notag\\
	&\leq C(\hat{p})+C(\hat{p})\mathbb{E}\left[ \left|\int_0^T \sigma_{i}^2(X_i(s))ds\right|^{\hat{p}/2}\right]\notag\\
	&\leq C(\hat{p}),
	\end{align*}
	which implies \eqref{bd2kq1}.  Next, for any  $s\leq t\leq t'\leq s'$, it follows from H\"older's inequality for integral that   
	\begin{equation} \label{bd2ul4} 
	|X_i(t')-X_i(t)|^{\hat{p}} \leq C(\hat{p})(s'-s)^{\hat{p}-1} \int_s^{s'} \sum\limits_{j\neq i} \dfrac{1}{|X_{ij}(u)|^{\hat{p}}}du +  C(\hat{p})(s'-s)^{\hat{p}} + 
	C(\hat{p})\Big| \int_t^{t'} \sigma_i(X_i(u))dW_i(u) \Big|^{\hat{p}}.
	\end{equation}
	By applying Theorem 1 in Fisher and Nappo (2009), we have 
	$$\mathbb{E}\left[\sup_{s\leq t, t'\leq s'} \left| \int_t^{t'}  \sigma_i(X_i(u))dW_i(u) \right|^{\hat{p}}\right] \leq C(\hat{p})\left(|s'-s|\ln\dfrac{2T}{|s'-s|}\right)^{\hat{p}/2},\mbox{ for all } 0\leq s<s'\leq T.$$
	This fact together with \eqref{bd2ul4} and Assumption $H_{\hat{p}}$ concludes \eqref{bd2kq3}.

(ii) Let that Assumption $H_{\hat{p}}$ hold for $\hat{p} \geq 3$.  Applying It\^o's formula, we have
\begin{align}
\dfrac{1}{X_i(t)-X_j(t)}&=-\int_0^t \dfrac{1}{X_{ij}^2(s)}\sum\limits_{k\neq i}\dfrac{\gamma_{ik}}{X_{ik}(s)}ds-\int_0^t \dfrac{b_i(X_i(s))}{X_{ij}^2(s)}ds-\int_0^t \dfrac{\sigma_i(X_i(s))}{X_{ij}^2(s)}dW_i(s)\notag\\
&+\int_0^t\dfrac{1}{X_{ij}^2(s)}\sum\limits_{\color{black} k\neq j}\dfrac{\gamma_{jk}}{X_{jk}(s)}ds+\int_0^t \dfrac{b_j(X_j(s))}{X_{ij}^2(s)}ds+\int_0^t \dfrac{\sigma_j(X_j(s))}{X_{ij}^2(s)}dW_j(s)\notag\\
&+\int_0^t \dfrac{\sigma_{i}^2(X_i(s))+\sigma_{j}^2(X_j(s))}{X_{ij}^3(s)}ds.\label{Xij}
\end{align}
By following a similar argument as in the proof of  \eqref{bd2kq1}, we obtain \eqref{bd2kq2}.
\end{proof}

	\begin{lemma}\label{Lem3}
		Let $S_{1i}$ and $R_1$ be defined by \eqref{dnS1i} and \eqref{dnR}, respectively. 
		{\begin{enumerate}
				\item[(i)] {\color{black}Let  Assumption $H_{\hat{p}}$ hold for some  $\hat{p}\geq 6$. For any $t_k^{(n)}\leq t\leq t_{k+1}^{(n)}$, it holds that  
				\begin{equation}
				R_1(t)\leq \dfrac{1}{n}\|e(t_k^{(n)})\|^2+\zeta_{1}(t)+\xi_1(t) , \label{ulR1}
				\end{equation}
				where $\zeta_1$ and $\xi_1$ are adapted processes satisfying $\bE\left[|\zeta_{1}(t)|\right]\leq \dfrac{C}{n^3}$ and $\bE\left[\xi_1(t)|\mathcal{F}_{t^{(n)}_k}\right]=0$. Moreover, it holds that
				\begin{equation}
				\sup_{t\in[0,T]}\mathbb{E}\left[{\| S_{1}(t)\|^2}\right]\leq \dfrac{C}{n^3}. \label{ulS1i}
				\end{equation}}
				\item[(iii)] Let Assumption $H_{\hat{p}}$ hold for some  $\hat{p}\geq 18$, then
				% there exists a positive constant $C$ 	which does not depend on $n$ such that
				\begin{equation}
				\mathbb{E}\left[\sup_{t\in[0,T]}{\| S_{1}(t)\|^2}\right]\leq \dfrac{C\log^3 n}{n^2}. \label{ulsupS1}
				\end{equation} 
		\end{enumerate}}
	\end{lemma}
	\begin{proof}
	(i) \ 	{\color{black}For each $t\in \left[t^{(n)}_k,t^{(n)}_{k+1}\right]$, by \eqref{dnS1i} and\eqref{Xij}, we can write
%	\begin{align*}
	$S_{1i}(t)=\bar{S}_{1i}(t)+\hat{S}_{1i}(t),$
%	\end{align*}
	where
	\begin{align*}
	\bar{S}_{1i}(t)&=\int_{t^{(n)}_k}^t\int_s^t\sum\limits_{j\neq i}\sum\limits_{k\neq i}\dfrac{\gamma_{ik}}{X_{ij}^2(u)X_{ik}(u)}duds-\int_{t^{(n)}_k}^t\int_s^t\sum\limits_{j\neq i}\sum\limits_{k\neq i}\dfrac{\gamma_{ik}}{X_{ij}^2(u)X_{jk}(u)}duds\\
	&+\int_{t^{(n)}_k}^t\int_s^t \sum\limits_{j\neq i} \dfrac{b_i(X_i(u))}{X_{ij}^2(u)}duds-\int_{t^{(n)}_k}^t\int_s^t \sum\limits_{j\neq i} \dfrac{b_j(X_j(u))}{X_{ij}^2(u)}duds\\
	&-\int_{t^{(n)}_{k}}^t\int_s^t\sum\limits_{j\neq i} \dfrac{1}{X_{ij}^3(u)}\left[\sigma_{i}^2(X_i(u))+\sigma_{j}^2(X_j(u))\right]duds,
	\end{align*}
	and
	\begin{align*}
	\hat{S}_{1i}(t)&=\int_{t^{(n)}_k}^t\int_s^t\sum\limits_{j\neq i}\dfrac{1}{X_{ij}^2(u)}\sigma_i(X_i(u))dW_i(u)ds-\int_{t^{(n)}_k}^t\int_s^t\sum\limits_{j\neq i}\dfrac{1}{X_{ij}^2(u)}\sigma_j(X_j(u))dW_j(u)ds\\
	&=\int_{t^{(n)}_k}^t\int_{t^{(n)}_k}^u\sum\limits_{j\neq i}\dfrac{1}{X_{ij}^2(u)}\sigma_i(X_i(u))dsdW_i(u)-\int_{t^{(n)}_k}^t\int_{t^{(n)}_k}^u\sum\limits_{j\neq i}\dfrac{1}{X_{ij}^2(u)}\sigma_j(X_j(u))dsdW_j(u),
	\end{align*}
	where the last equation follows from Fubini's theorem.
	Using Holder's inequality and the estimate $x^4y^2\leq \dfrac{2}{3}x^6+\dfrac{1}{3}y^6$, we have 
\begin{align*}
	\bar{S}_{1i}^2(t)&\leq \dfrac{C}{n^2}\left\{\int_{t^{(n)}_k}^t\int_s^t\sum\limits_{j\neq i}\sum\limits_{k\neq i}\dfrac{\gamma_{ik}^2}{X_{ij}^4(u)X_{ik}^2(u)}duds+\int_{t^{(n)}_k}^t\int_s^t\sum\limits_{j\neq i}\sum\limits_{k\neq i}\dfrac{\gamma_{ik}^2}{X_{ij}^4(u)X_{jk}^2(u)}duds\right.\\
	&+\int_{t^{(n)}_k}^t\int_s^t \sum\limits_{j\neq i} \dfrac{b_i^2(X_i(u))}{X_{ij}^4(u)}duds+\int_{t^{(n)}_k}^t\int_s^t \sum\limits_{j\neq i} \dfrac{b_j^2(X_j(u))}{X_{ij}^4(u)}duds\\
	&\left.+\int_{t^{(n)}_{k}}^t\int_s^t\sum\limits_{j\neq i} \dfrac{\sigma_{i}^4(X_i(u))+\sigma_{j}^4(X_j(u))}{X_{ij}^6(u)}duds.\right\}\\
	&\leq \dfrac{C}{n^2}\left\{\int_{t^{(n)}_{k}}^t\int_s^t\sum\limits_{i=1}^{d-1}\left[\dfrac{1}{X_{i,i+1}^6(u)}+\dfrac{1}{X_{i,i+1}^4(u)}\right]duds\right\}.
\end{align*}
Set
$$\zeta_1(t)=\dfrac{n}{4}\sum\limits_{i=1}^{d}\bar{S}_{1i}^2(t)\mbox{ and }\xi_1(t)=\sum\limits_{i=1}^{d}e_i(t^{(n)}_k)\hat{S}_{1i}(t).$$
By Young's inequality, 
		\begin{align*}
		R_1(t)&=\sum\limits_{i=1}^{d}e_i(t_k^{(n)})S_{1i}(t)\leq \sum\limits_{i=1}^{d}\left(\dfrac{e_i^2(t_k^{(n)})}{n}+\dfrac{n}{4}\bar{S}_{1i}^2(t)\right)+\sum\limits_{i=1}^{d}e_i(t^{(n)}_k)\hat{S}_{1i}(t)\\
		&= \dfrac{1}{n}\|e(t_k^{(n)})\|^2+\zeta_{1}(t)+\xi_1(t).
		\end{align*}
		Since Assumption $H_{\hat{p}}$ holds for $\hat{p}\geq 6$,  $\bE\left[|\zeta_1(t)|\right]\leq \dfrac{C}{n^3}$. Moreover, it easy to see that $\bE\left[\xi_1(t)|\mathcal{F}_{t^{(n)}_k}\right]=0$. Thus we can conclude \eqref{ulR1}. }
		
		Next, using the AM-GM inequality $a+b+c \geq 3\sqrt[3]{abc}$ for non-negative numbers $a,b,c$, and the H\"older inquality for integral, we get 
		\begin{align*}
		{\| S_{1}(t)\|^2}&\leq \dfrac{C}{n^2}\sum\limits_{i=1}^{d}\sum\limits_{j\neq i}\left\{\int_{t^{(n)}_{k}}^t n^3\left|X_i(s)-X_i(t)\right|^6ds+\int_{t^{(n)}_{k}}^t n^3\left|X_j(s)-X_j(t)\right|^6ds\right. \notag \\
		&\left. +2\int_{t^{(n)}_{k}}^t\left|X_{ij}(s)\right|^{-6}ds+2\int_{t^{(n)}_{k}}^t\left|X_{ij}(t)\right|^{-6}ds \right\}.
		\end{align*}
		This estimate together with Assumption $H_{\hat{p}}$ for  $\hat{p}\geq 6$ implies \eqref{ulS1i}. 
		
		(ii) \ Finally we show \eqref{ulsupS1}. Note that $\sup_{t\in[0,T]} {\| S_{1}(t)\|^2}$ is bounded above by  
		\begin{align*}
		  \dfrac{C}{n^2}\sum\limits_{k=0}^{n-1}\sum\limits_{i=1}^{d}\sum\limits_{j\neq i}\left\{\sup_{t^{(n)}_k\leq t\leq t'\leq t^{(n)}_{k+1}}n^2\|X(t)-X(t')\|^6+\dfrac{1}{n}\max_{i=1,\ldots,d}\sup_{t^{(n)}_k\leq t\leq t'\leq t^{(n)}_{k+1}}|X_{i,i+1}(t)|^{-6} \right\}.
		\end{align*}
		 If    Assumption $H_{\hat{p}}$ holds with $\hat{p} \geq 18$, then using Lemma \ref{bd2} we obtain \eqref{ulsupS1}. 
	\end{proof}
	\begin{lemma}\label{Lem4}
		Let $S_{2i}$ and $R_2$ be defined by \eqref{dnS2i} and \eqref{dnR}, respectively. Let Assumption $H_{\hat{p}}$ hold for some  $\hat{p}\geq 2$, for any $t_k^{(n)}\leq t\leq t_{k+1}^{(n)}$, then 
%		\begin{enumerate}
	%			\item[(i)] It holds that
				{\color{black}\begin{equation}
				R_2(t)\leq \dfrac{1}{n}\|e(t_k^{(n)})\|^2+\zeta_{2}(t)+\xi_2(t),  \label{ulR2}
				\end{equation}
				where $\zeta_2$ and $\xi_2$ are adapted processes satisfying $\bE\left[\|\zeta_{2}(t)\|\right]\leq \dfrac{C}{n^3}$ and $\bE\left[\xi_2(t)|\mathcal{F}_{t^{(n)}_k}\right]=0$. Moreover, it holds that} 
				\begin{equation}
				\sup_{t\in[0,T]}\mathbb{E}[{\| S_{2}(t)\|^2}]\leq \dfrac{C}{n^3}, \label{ulS2}
				\end{equation}	
				and 
				\begin{equation}
				\mathbb{E}\left[\sup_{t\in[0,T]} {\| S_{2}(t)\|^2}\right]\leq \dfrac{C}{n^2}. \label{ulsupS2}
				\end{equation}
	\end{lemma}

\begin{proof}
	For each $t\in \left[t_k^{(n)},t^{(n)}_{k+1}\right]$, applying It\^{o}'s formula, we have
	\begin{align*}
	b_i(X_i(t))&-b_i(X_i(t^{(n)}_k))=\int_{t^{(n)}_k}^t b_i'(X_i(u))\sigma_i(X_i(u))dW_i(u)\\
	&+ \int_{t^{(n)}_k}^t \left(b_i'(X_i(u))\sum\limits_{j\neq i}\dfrac{\gamma_{ij}}{X_{ij}(u)}+b_i'(X_i(u))b_i(X_i(u))+\dfrac{b_i''(X_i(u))\sigma_i^2(X_i(u))}{2}\right)du.
	\end{align*}
	Combine with \eqref{dnS2i}, we can write $S_{2i}(t)=\bar{S}_{2i}(t)+\hat{S}_{2i}(t),$ where
	$$\bar{S}_{2i}(t)=\int_{t^{(n)}_k}^t\int_{t^{(n)}_k}^s\left(b_i'(X_i(u))\sum\limits_{j\neq i}\dfrac{\gamma_{ij}}{X_{ij}(u)}+b_i'(X_i(u))b_i(X_i(u))+\dfrac{b_i''(X_i(u))\sigma_i^2(X_i(u))}{2}\right)duds,$$
	and
	$$\hat{S}_{2i}(t)=\int_{t^{(n)}_k}^t\int_{t^{(n)}_k}^sb_i'(X_i(u))\sigma_i(X_i(u))dW_i(u)ds.$$
	Set
	$\zeta_2(t)=\dfrac{n}{4}\sum\limits_{i=1}^{d}\bar{S}_{2i}^2(t),\mbox{and }\xi_2(t)=\sum\limits_{i=1}^{d}e_i(t^{(n)}_k)\hat{S}_{2i}(t).$
	By a similar argument as in proof of Lemma \ref{Lem3}, we obtain \eqref{ulR2}, \eqref{ulS2}, and \eqref{ulsupS2}.
\end{proof}
	\begin{lemma}\label{Lem5}
		Let $S_{3i}$ and $R_3$ be defined by \eqref{dnS3i} and \eqref{dnR}, respectively. 
				 Let Assumption $H_{\hat{p}}$ hold for some  $\hat{p}\geq 2$, then 
				\begin{equation}
				\sup_{t\in[0,T]}\mathbb{E}\left[{\| S_{3}(t)\|^2} \right]\leq \dfrac{C}{n^3}, \label{ulS3}
				\end{equation} 
			and
				\begin{equation}
				\mathbb{E}\left[\sup_{t\in[0,T]}{\| S_{3}(t)\|^2}\right]\leq \dfrac{C}{n^2}. \label{ulsupS3}
				\end{equation}
	\end{lemma}
	\begin{proof}
		For each $i=1,\ldots, d$, applying It\^o's formula for $\sigma_i$, we get  
		\begin{align*}
		\sigma_i(X_i(s))&-\sigma_i(X_i(t_k^{(n)}))-\int_{t^{(n)}_{k}}^s{\color{black}\sigma_i'(X_i(t_k^{(n)}))\sigma_i(X_i(t_k^{(n)}))}dW_i(u)\\
		=&\int_{t^{(n)}_k}^s\left[\sum\limits_{j\neq i} \dfrac{\gamma_{ij}\sigma_i'(X_i(u))}{X_{ij}(u)}+\sigma_i'(X_i(u)) b_i(X_i(u))+\dfrac{1}{2}\sigma_i''(X_i(u))\sigma_i^2(X_i(u))\right]du\\
		&+\int_{t^{(n)}_{k}}^s\left[{\color{black}\sigma_i'(X_i(u))\sigma_i(X_i(u))-\sigma_i'(X_i(t_k^{(n)}))\sigma_i(X_i(t_k^{(n)}))}\right]dW_i(u).
		\end{align*}
		Using Doob's maximal inequality and H\"older's inequality for integral, we get 
		\begin{align}
		\mathbb{E}&\left\{\sup_{t^{(n)}_k\leq t\leq t^{(n)}_{k+1}}\left[\int_{t^{(n)}_k}^t\int_{t^{(n)}_k}^s\left(\sum\limits_{j\neq i} \dfrac{\gamma_{ij}\sigma_i'(X_i(u))}{X_{ij}(u)}+{\color{black}\sigma_i'(X_i(u))b_i(X_i(u))+\dfrac{1}{2}\sigma_i''(X_i(u))\sigma_i^2(X_i(u))}\right)dudW_i(s)\right]^2\right\} \notag \\
		&\leq 4\mathbb{E}\left[\int_{t^{(n)}_k}^{t^{(n)}_{k+1}}\left(\int_{t^{(n)}_k}^s\left|\sum\limits_{j\neq i} \dfrac{\gamma_{ij}\sigma_i'(X_i(u))}{X_{ij}(u)}+{\color{black}\sigma_i'(X_i(u))b_i(X_i(u))+\dfrac{1}{2}\sigma_i''(X_i(u))\sigma_i^2(X_i(u))}\right|du\right)^2ds\right] \notag \\
		&\leq \dfrac{C}{n} \mathbb{E}\left[\int_{t^{(n)}_k}^{t^{(n)}_{k+1}}\int_{t^{(n)}_k}^{t^{(n)}_{k+1}}\left|\sum\limits_{j\neq i} \dfrac{\gamma_{ij}\sigma_i'(X_i(u))}{X_{ij}(u)}+{\color{black}\sigma_i'(X_i(u))b_i(X_i(u))+\dfrac{1}{2}\sigma_i''(X_i(u))\sigma_i^2(X_i(u))}\right|^2duds\right] \notag \\
		&\leq \dfrac{C}{n^3}, \label{ulsupW}
		\end{align}
		where the last estimate follows from Assumption $H_{\hat{p}}$ and the fact that $b, \sigma \in C^2_b$. Similary, by using Doob's maximal inequality, the It\^o isometry and the Lipschitz continuity of $\sigma_i'\sigma_i$,  we get 		
		\begin{align*}
		\mathbb{E}&\left[\sup_{t^{(n)}_k\leq t\leq t^{(n)}_{k+1}}\left(\int_{t^{(n)}_{k}}^t\int_{t^{(n)}_{k}}^s\left[{\color{black}\sigma_i'(X_i(u))\sigma_i(X_i(u))-\sigma_i'(X_i(t_k^{(n)}))\sigma_i(X_i(t_k^{(n)}))}\right]dW_i(u)dW_i(s)\right)^2\right]\\
		&\leq C\mathbb{E}\left[\int_{t^{(n)}_{k}}^{t^{(n)}_{k+1}}\left(\int_{t^{(n)}_{k}}^s\left[{\color{black}\sigma_i'(X_i(u))\sigma_i(X_i(u))-\sigma_i'(X_i(t_k^{(n)}))\sigma_i(X_i(t_k^{(n)}))}\right]dW_i(u)\right)^2ds\right]\\
		&=C\int_{t^{(n)}_{k}}^{t^{(n)}_{k+1}}\int_{t^{(n)}_{k}}^s\mathbb{E}\left[ \left| {\color{black}\sigma_i'(X_i(u))\sigma_i(X_i(u))-\sigma_i'(X_i(t_k^{(n)}))\sigma_i(X_i(t_k^{(n)}))}\right|^2 \right] duds\\
		&\leq C\int_{t^{(n)}_{k}}^{t^{(n)}_{k+1}}\int_{t^{(n)}_{k}}^s\mathbb{E}\left[|X_i(u)-X_i(t_k^{(n)})|^2\right]duds\\
		&\leq \dfrac{C}{n^3},
		\end{align*}
where the last estimate follows from Assumption 2.1. This estimate together with \eqref{ulsupW} implies that  $$\sup_{0\leq k < n} \mathbb{E}\left[\sup_{t_k^{(n)}\leq t\leq  t_{k+1}^{(n)}} {\| S_{3}(t)\|^2}\right]\leq \dfrac{C}{n^3},$$
		which concludes \eqref{ulS3}. 
 		Moreover,  
 		\begin{equation*}
		\mathbb{E}\left[\sup_{t\in [0,T]} {\| S_{3}(t)\|^2}\right]\leq \sum\limits_{k=0}^{n-1}\mathbb{E}\left[\sup_{t_k^{(n)}\leq t\leq  t_{k+1}^{(n)}}{\| S_{3}(t)\|^2}\right],
		\end{equation*}
		which implies \eqref{ulsupS3}. 
	\end{proof}
	\begin{lemma}\label{Lem6}
		Let $S_{4i}$ and $R_4$ be defined by \eqref{dnS4i} and \eqref{dnR}, respectively.  
		\begin{enumerate}
				\item[(i)] It holds  
				\begin{equation*}
				{\| S_{4}(t)\|^2}\leq \dfrac{C}{n^2}\|e(t_k^{(n)})\|^2\text{ and } R_4(t)\leq \dfrac{C}{n}\|e(t_k^{(n)})\|^2 \text{ for any } t_k^{(n)}\leq t\leq t_{k+1}^{(n)}, \label{ulR4}
				\end{equation*}	
				\item[(ii)] Moreover, 
				$$\mathbb{E}\left[\sup_{t\in[0,T]} {\| S_{4}(t)\|^2}\right]\leq \dfrac{C}{n^2}\sum\limits_{k=0}^{n}\mathbb{E}{\color{black}\left[\|e(t_k^{(n)})\|^2\right]}.$$
		\end{enumerate}
	\end{lemma}
	\begin{proof}
		These estimates follows from the  Lipschitz property of $b_i(x)$, so, we skip the detailed proof. 
	\end{proof}
	
		\begin{lemma}\label{Lem7}
			Let $S_{5i}$ and $R_5$ be defined by \eqref{dnS5i} and \eqref{dnR}, respectively.    
			{\begin{enumerate}
					\item[(i)] For any $t_k^{(n)}\leq t\leq t_{k+1}^{(n)}$, it holds  
					\begin{align} \label{ulS5} 
					{\| S_{5}(t)\|^2}\leq \dfrac{C}{n}\|e(t_k^{(n)})\|^2+\xi_5(t),
					\end{align} 
	where $\xi_5(t)$ is an adapted process satisfying $\mathbb{E}{\color{black}\left[\xi_5(t)|\mathcal{F}_{t_k^{(n)}}\right]}=0$.
					\item[(ii)] Moreover, 
					\begin{align} \label{ulS5tau} \bE\left[\sup_{t\in [0,T]} \| S_{5}(t)\|^2 \right]\leq \dfrac{C}{n}\sum\limits_{k=0}^{n}\mathbb{E}\left[\|e(t_k^{(n)})\|^2\right].
					\end{align}
			\end{enumerate}}
	\end{lemma}
	\begin{proof}
		%Consider $\eta_n(t)=t_k^{(n)}$ if 
For any $t\in \left[t^{(n)}_k,t^{(n)}_{k+1}\right]$, we have
		\begin{align*}
		 \| S_{5}(t)\|^2&{\color{black}\leq C\sum\limits_{i=1}^d \|e(t_k^{(n)})\|^2(W_i(t)-W_i (t_k^{(n)}))^2}\\
		&\leq \dfrac{C}{n}\|e(t_k^{(n)})\|^2+ C\sum_{i=1}^d \|e(t_k^{(n)})\|^2\left[(W_i(t)-W_i(t_k^{(n)}))^2-(t-t_k^{(n)})\right],
		\end{align*}
		which implies \eqref{ulS5} with $\xi_5(t)=C\sum_{i=1}^d \|e(t_k^{(n)})\|^2\left[(W_i(t)-W_i(t_k^{(n)}))^2-(t-t_k^{(n)})\right]$.  
On the other hand,  
		\begin{align*}
\sup_{t\in [0,T]} \| S_{5}(t)\|^2&{\color{black}\leq C\sum\limits_{i=1}^d \sum_{k=0}^n  \|e(t_k^{(n)})\|^2 \sup_{t_k^{(n)} \leq t \leq t_{k+1}^{(n)}} (W_i(t) -W_i (t_k^{(n)}))^2}.
\end{align*}
Note that  $\sup_{t_k^{(n)} \leq t \leq t_{k+1}^{(n)}} (W_i(t) -W_i (t_k^{(n)}))^2$ and $e(t_k^{(n)})$ are independent, we have 
		\begin{align*}
\bE\left[\sup_{t\in [0,T]} \| S_{5}(t)\|^2 \right] &\leq C\sum\limits_{i=1}^d \sum_{k=0}^n \bE\left[ \|e(t_k^{(n)})\|^2\right] \bE \left[ \sup_{t_k^{(n)} \leq t \leq t_{k+1}^{(n)}} (W_i(t) -W_i (t_k^{(n)}))^2\right] \\
&\leq 4C\sum\limits_{i=1}^d \sum_{k=0}^n \bE\left[ \|e(t_k^{(n)})\|^2\right] \bE \left[  (W_i(t^{(n)}_{k+1}) -W_i (t_k^{(n)}))^2\right]\\
&= \frac{4Cd}{n} \sum_{k=0}^n \bE\left[ \|e(t_k^{(n)})\|^2\right],
\end{align*}
which implies \eqref{ulS5tau}. 
	\end{proof}

\begin{lemma}\label{Lem8}   Let $S_{6i}$  be defined by  \eqref{dnS6i}. 
	\begin{enumerate}
			\item[(i)] For any $ t \in [ t_k^{(n)}, t_{k+1}^{(n)}]$, it holds
			\begin{equation}
			\| S_{6}(t)\|^2\leq \dfrac{C}{n^2}\|e(t_k^{(n)})\|^2+\xi_6(t), \label{ulS6}
			\end{equation}
			where $\xi_6(t)$ is an adapted process satisfying $\mathbb{E}{\color{black}\left[\xi_6(t)|\mathcal{F}_{t_k^{(n)}}\right]}=0$.
			\item[(ii)] Moreover, 
			\begin{align} \label{ulS6tau} 
			\bE\left[\sup_{t \in [0,T]}\| S_{6}(t)\|^2\right]
			\leq \dfrac{C}{n^2}\sum\limits_{k=0}^{n}\mathbb{E}{\color{black}\left[\|e(t_k^{(n)})\|^2\right]}.
			\end{align}
%			for all stopping times $\tau$ taking value in $[0,T]$.
	\end{enumerate}
\end{lemma}
\begin{proof}
	For any {\color{black}$t\in\left[t_k^{(n)},t_{k+1}^{(n)}\right]$}, since $\sigma\sigma'$ is Lipschitz continuous, we have
	\begin{align*}
	\| S_{6}(t)\|^2&\leq C\sum\limits_{i=1}^d \|e(t_k^{(n)})\|^2\left[(W_i(t)-W_i (t_k^{(n)}))^2-(t-t_k^{(n)})\right]^2\\
	&\leq \dfrac{C}{n^2}\|e(t_k^{(n)})\|^2\\
	& \quad + C\sum_{i=1}^d \|e(t_k^{(n)})\|^2\left\{\left[(W_i(t)-W_i(t_k^{(n)}))^2-(t-t_k^{(n)})\right]^2-2(t-t_k)^2\right\}.
	\end{align*}
	%with $\mathbb{E}\left[\left[(W_i(t)-W_i(t_k^{(n)}))^2-(t-t_k^{(n)})\right]^2-2(t-t_k)^2\right]=0$.\\
	Set $\xi_6(t)=\sum_{i=1}^d \|e(t_k^{(n)})\|^2\left\{\left[(W_i(t)-W_i(t_k^{(n)}))^2-(t-t_k^{(n)})\right]^2-2(t-t_k)^2\right\}$, we have \eqref{ulS6} and $\mathbb{E}{\color{black}\left[\xi_6(t)|\mathcal{F}_{t_k^{(n)}}\right]}=0$. 	On the other hand,  we have  
	\begin{align*}
		\sup_{t \in [0,T]}\| S_{6}(t)\|^2&\leq C \sum\limits_{k=0}^n  \sum\limits_{i=1}^d \|e(t_k^{(n)})\|^2 \sup_{t^{(n)}_k \leq t \leq t^{(n)}_{k+1}}\left[(W_i(t)-W_i (t_k^{(n)}))^2-(t-t_k^{(n)})\right]^2.
	\end{align*} 
	Note that  $\sup_{t^{(n)}_k \leq t \leq t^{(n)}_{k+1}}\left[(W_i(t)-W_i (t_k^{(n)}))^2-(t-t_k^{(n)})\right]^2$ and $\|e(t_k^{(n)})\|^2$ are independent, we have 
	\begin{align*}
	&\bE\left[\sup_{t \in [0,T]}\| S_{6}(t)\|^2\right]\\ 
	&\leq C \sum\limits_{k=0}^n  \sum\limits_{i=1}^d \bE\left[ \|e(t_k^{(n)})\|^2 \right] \bE \left[  \sup_{t^{(n)}_k \leq t \leq t^{(n)}_{k+1}}\left[(W_i(t)-W_i (t_k^{(n)}))^2-(t-t_k^{(n)})\right]^2\right]\\
	&\leq 4C \sum\limits_{k=0}^n  \sum\limits_{i=1}^d \bE\left[ \|e(t_k^{(n)})\|^2 \right] \bE \left[  \left[(W_i(t^{(n)}_{k+1})-W_i (t_k^{(n)}))^2-(t^{(n)}_{k+1}-t_k^{(n)})\right]^2\right]\\
	&= \dfrac{8C}{n^2}\sum_{k=0}^n \|e(t_k^{(n)})\|^2,	\end{align*} 
	where the last inquality follows from Dood's maximal inequality for non-negative sub-martingale. 
Therefore, we conclude \eqref{ulS6tau}.
\end{proof}

\subsection{Proof of Theorem \ref{Thm1}}
	(i) \  Let Assumption $H_{\hat{p}}$ hold for $\hat{p} \geq 6$. For each $t \in \left[t^{(n)}_k, t^{(n)}_{k+1}\right]$,  it follows from  \eqref{et} and Lemmas \ref{Lem3} --  \ref{Lem7} that 
	{\color{black}\begin{align*}
	\|e(t)\|^2& \leq \left(1+\dfrac{C}{n}\right)\|e(t_k^{(n)})\|^2+\|S_1(t)\|^2+\|S_2(t)\|^2+\|S_3(t)\|^2+\zeta_1(t)+\zeta_2(t)\\
		&+\xi_1(t)+\xi_2(t)+R_3(t)+R_5(t)+\xi_5(t)+\xi_6(t)+R_6(t)\notag\\
	&\leq \left(1+\dfrac{C}{n}\right)\|e(t_k^{(n)})\|^2+\zeta(t)+\xi(t),
	\end{align*}
	where  $\zeta(t)=\|S_1(t)\|^2+\|S_2(t)\|^2+\|S_3(t)\|^2+\zeta_1(t)+\zeta_2(t)$, and $\xi(t)=\xi_1(t)+\xi_2(t)+\xi_5(t)+\xi_6(t)+R_3(t)+R_5(t)+R_6(t)$. It also  follows from Lemmas \ref{Lem3} --  \ref{Lem5} that 
	$$\sup_{t\in[0,T]}\mathbb{E}[|\zeta(t)|]\leq \dfrac{C}{n^3}.$$}
%	Moreover, $\mathbb{E}(\xi(t)| \mathcal{F}_{t_k^{(n)}})=0$.
	By choosing $t=t_{k+1}^{(n)}$, we have 
	$$\|e(t_{k+1}^{(n)})\|^2\leq \left(1+\dfrac{C}{n}\right)\|e(t_k^{(n)})\|^2+\zeta(t_{k+1}^{(n)})+\xi(t_{k+1}^{(n)}).$$
	Moreover, since  $\mathbb{E}\left[\xi(t^{(n)}_{k+1})| \mathcal{F}_{t_k^{(n)}}\right]=0$, by applying Lemma \ref{Lem1} with $q = 1+ \frac{C}{n}$, we obtain \eqref{kq1}. 
	
	The estimate \eqref{kq2} is a consequence of \eqref{kq1} and Lemma 3.2 in Gy\"ongy and Krylov (2003).  
	
\vskip 0.2cm 
\noindent 
	(ii) \ Suppose that Assumption $H_{\hat{p}}$ holds for $\hat{p} \geq 18$, we show \eqref{kq3}. From \eqref{et}, we have 
	$$\sup_{t \in [0,T]} \|e(t)\|^2\leq C{\sup_{0\leq k \leq n}\|e(t_k^{(n)})\|^2}+C\sum\limits_{m=1}^{6}{\sup_{t\in[0,T]}\|S_m(t)\|^2}.$$
	If  {\color{black} $p\in (0,2)$}, applying the simple estimate $\left( \sum_j a_j^2 \right)^{p/2} \leq \sum_j |a_j|^p$, we get 
	\begin{align*}
		\bE \left[ \sup_{t \in [0,T]} \|e(t)\|^p\right] \leq &C \bE \left[ {\sup_{0\leq k \leq n}\|e(t_k^{(n)})\|^p}\right] +C\sum\limits_{m=1}^{6} \bE \left[ {\sup_{t\in[0,T]}\|S_m(t)\|^p}\right]\\
		\leq &C \bE \left[ {\sup_{0\leq k \leq n}\|e(t_k^{(n)})\|^p}\right] +C\sum\limits_{m=1}^{6} \left( \bE \left[ {\sup_{t\in[0,T]}\|S_m(t)\|^2}\right]\right)^{p/2}.
	\end{align*} 

	This estimate together with Lemmas \ref{Lem3} -- \ref{Lem8} concludes \eqref{kq3}.

\section{Example} \label{sec4} 
In this section, we present a numerical example to justify our asymptotic  convergence analysis of the semi-implicit Milstein (SIM) scheme. We also compare this scheme with the semi-implicit  Euler-Maruyama (SIEM) scheme in \cite{NT2017}.  We consider a system of Brownian particles with nearest neighbor repulsion   $X=(X_1,\ldots,X_d)$  given by the following SDEs
\begin{equation} \label{vdX}
\begin{cases} \rd X_1(t) = \left\{ \frac{\gamma}{X_1(t)-X_2(t)} + b_1(X_1(t))\right\} \rd t +  \sigma_1 (X_1(t))\rd W_1(t),\\
\rd X_i(t) = \left\{ \frac{\gamma}{X_i(t)-X_{i-1}(t)} +\frac{\gamma}{X_i(t)-X_{i+1}(t)} + b_i(X_i(t))\right\}\rd t + \sigma_{i}(X_i(t))\rd W_i(t),\\
\hspace{9cm} i=2,\ldots, d-1,\\
\rd X_d(t) = \left\{ \frac{\gamma}{X_d(t)-X_{d-1}(t)} + b_d(X_{d}(t))\right\} \rd t +   \sigma_{d}(X_d(t))\rd W_d(t),
\end{cases}
\end{equation} 
with $X(0) \in \Delta_d$.  In particular, we choose $d = 10$, $\gamma = 1$, and  $b_i(x)=\sin x$, $\sigma_i(x)= \frac { \sin(2x)}{2}$ for $i=1,\ldots, 10$. 
Then it follows from Corollary 6.2 in \cite{GrMa14} that equation \eqref{vdX} has a unique strong solution in $\Delta_d$ for all $t>0$. 

Let's denote by $X^{\mathcal{E},n}$ and $X^{\mathcal{M},n}$ the SIEM and SIM approximate solutions of $X$, respectively. Let's also denote by  
$(X^{\mathcal{E},n,i})_{i \geq 1}$ and $(X^{\mathcal{M}, n,i})_{i\geq 1}$  independent and identicaly distributed copies of random variables $X^{\mathcal{E},n}$ and $X^{\mathcal{M},n}$, respectively. The iteration method in \cite{NT2017}, Propostion 4.1, is applied to estimate solution to systems of algebraic equations of the type \eqref{Xnt} for both schemes. 
 We use
$$mse^{\mathcal{E}}(k) =  \frac{1}{M}\sum_{i=1}^M \| X^{\mathcal{E},2^k,i}(1) -   X^{\mathcal{E},2^{k+1},i}(1)\|^2,$$
and 
$$mse^{\mathcal{M}}(k) =  \frac{1}{M}\sum_{i=1}^M \| X^{\mathcal{M},2^k,i}(1) -   X^{\mathcal{M},2^{k+1},i}(1)\|^2,$$
to estimate the convergence rate. It is justified by the fact that if a scheme $X^{\bullet,n}$ converges at the rate of order $\beta \in (0,+\infty)$ in $L^2$-norm then  there exists some constant $\beta >0$ such that 
$$2^{2\beta n} \bE \left[ \| X(1) - X^{\bullet,2^n}(1)\|^2 \right] = O(1),$$
then also 
$$ \ 2^{2\beta n} \bE\left[ \| X^{\bullet,2^{n+1}}(1) - X^{\bullet,2^n}(1)\|^2\right] = O(1),$$
and vice versa. In this case, we can write 
$$\log_2(mse^\bullet(k)) = -2\beta k + \tilde{\beta} + o(1),$$
for some $\tilde \beta \in \mathbb R$.  
Figure \ref{Fig1} shows the simulation result where we compute $mse^{\mathcal E}(k)$ and $mse^{\mathcal M}(k)$ for $k=1,\ldots, 5$, $M = 10^3$ and $X_0(i) = i/2$ (left), $X_0(i) = i$ (center), and $X_0(i) = 2i$ (right). We draw regression lines to estimate the rates of convergence $\beta$ for each scheme. 

\begin{figure} 
\begin{center} 
\includegraphics[width = 4cm]{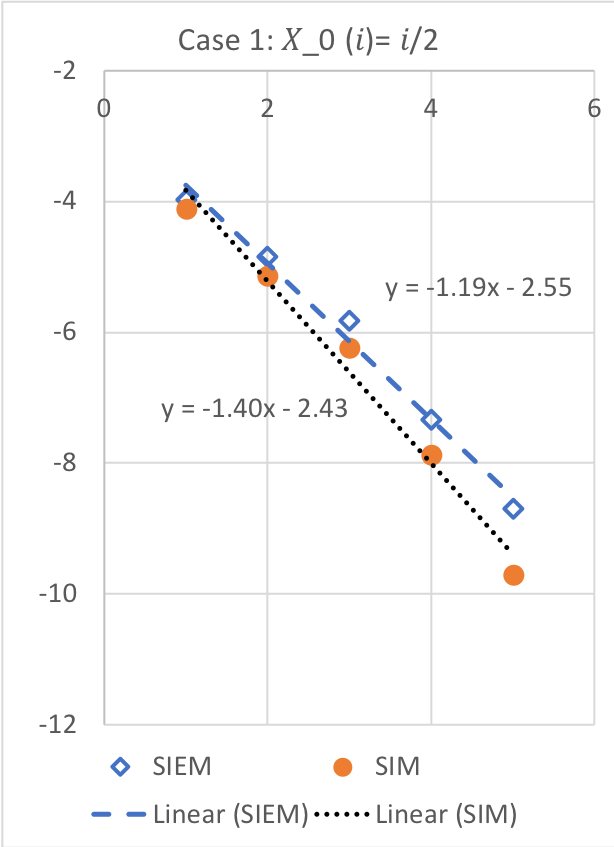}
\includegraphics[width = 4cm]{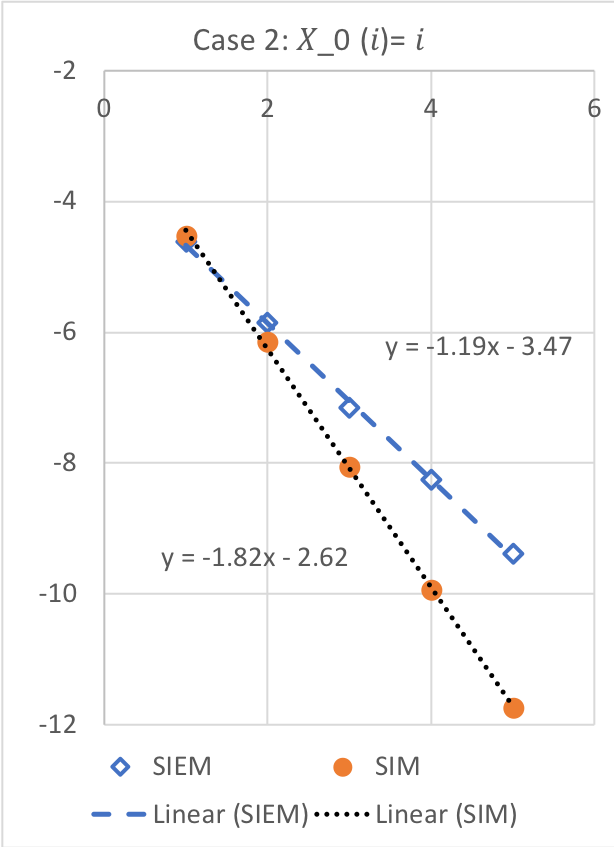}
\includegraphics[width = 4cm]{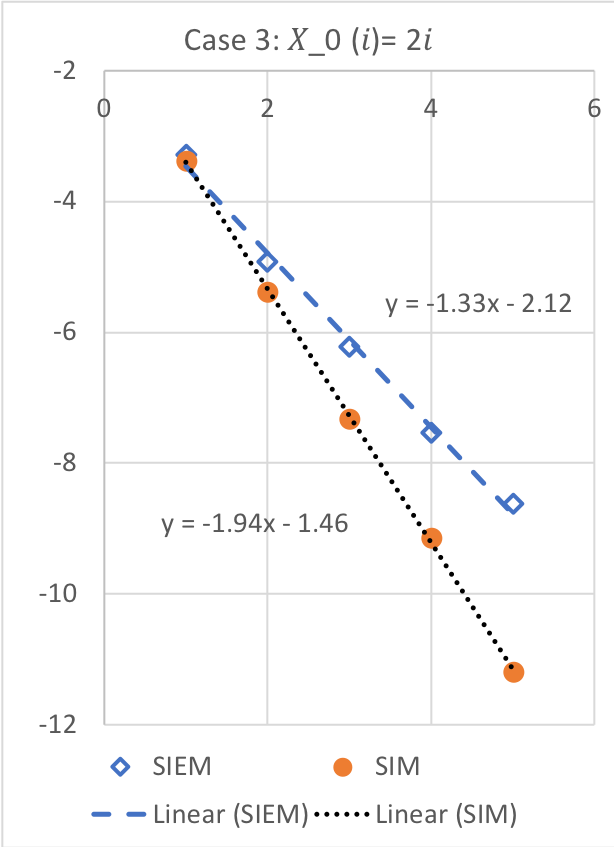}
\caption{Values of $mse^{\mathcal E}(k)$ (marked by $\diamond$) and  $mse^{\mathcal M}(k)$ (marked by $\bullet$) in $\log_2$-scale with $k=1,\ldots, 5$}
\label{Fig1} 
\end{center} 
\end{figure}
The empirical rates of convergence for each numerical scheme are shown in Table \ref{Table1}.   
We see that the rate of convergence of each scheme seems to depend on the distances between particles at the initial time. However, the rate of convergence of the SIM scheme is always higher than the one of SIEM scheme. In Cases 2 and 3, the rates of SIM scheme are close to 1, which supports our theoretical result. However, the rate of SIM scheme in Case 1 is low, which may due to the fact that the iteration method in Proposition 4.1 of \cite{NT2017} converges slowly in this case.  
\begin{table}
\begin{center} 
\begin{tabular}{c|ccc}
	 & Case 1: $X_i(0) =i/2$   & Case 2: $X_i(0) =i$ & Case 3: $X_i(0) =2i$ \\
	 \hline 
	SIEM scheme & 0.59 & 0.59 & 0.66\\
	SIM scheme & 0.70 & 0.91 & 0.97
\end{tabular}
\caption{Empirical rates of convergence of each numerical scheme}
\label{Table1}
\end{center} 
\end{table}

\end{document}